\documentclass[a4paper]{amsart}
\usepackage{color}
\usepackage[latin1]{inputenc}
\usepackage[T1]{fontenc}
\usepackage{amsfonts}
\usepackage{amssymb}
\usepackage{amsmath}
\usepackage{amsthm}
\usepackage{stmaryrd}
\usepackage{enumerate}
\usepackage{multirow}
\usepackage{graphicx}

\usepackage{graphicx,type1cm,eso-pic,color}
\usepackage{pstricks}
\usepackage{float}
\newtheorem{theorem}{Theorem}[section]
\newtheorem{lemma}[theorem]{Lemma}
\newtheorem{proposition}[theorem]{Proposition}
\newtheorem{corollary}[theorem]{Corollary}

\newtheorem{assumption}[theorem]{Assumption}

\newtheorem{example}[theorem]{Example}
\begin{document}
\setlength\arraycolsep{2pt}
\title[Projection Estimators of the Stationary Density of a fDE]{Projection Estimators of the Stationary Density of a Differential Equation Driven by the Fractional Brownian Motion}
\author{Nicolas MARIE$^{\dag}$}
\address{$^{\dag}$Laboratoire Modal'X, Universit\'e Paris Nanterre, Nanterre, France}
\email{nmarie@parisnanterre.fr}
\date{}
\maketitle
\noindent
%

% Abstract.

%
\begin{abstract}
The paper deals with projection estimators of the density of the stationary solution $X$ to a differential equation driven by the fractional Brownian motion under a dissipativity condition on the drift function. A model selection method is provided and, thanks to the concentration inequality for Lipschitz functionals of discrete samples of $X$ proved in Bertin et al. (2020), an oracle inequality is established for the adaptive estimator.
\end{abstract}
\tableofcontents
%

% Section : Introduction.

%
\section{Introduction}\label{section_introduction}
Consider the differential equation
\begin{equation}\label{main_equation}
X_t = X_0 +\int_{0}^{t}b(X_s)ds +\sigma B_t
\textrm{ $;$ }t\in\mathbb R_+,
\end{equation}
where $X_0$ is a real-valued random variable, $B = (B_t)_{t\in\mathbb R_+}$ is a fractional Brownian motion of Hurst index $H\in (0,1)$, $b :\mathbb R\rightarrow\mathbb R$ is a continuous map and $\sigma\in\mathbb R^* :=\mathbb R\backslash\{0\}$. Throughout the paper, it is assumed that Equation (\ref{main_equation}) has a unique stationary solution, which means in particular that there exists a unique random variable $X_0$ such that $X_t$ has the same distribution than $X_0$ for every $t\in\mathbb R_+$. A sufficient condition is given at Section \ref{section_preliminaries}.
\\
\\
For over two decades, many authors have investigated statistical questions related to differential equations driven by the fractional Brownian motion (fDE).\\
A large part of the papers published on statistical inference for fDEs deals with parametric estimators of the drift function $b$ when $H > 1/2$. In Kleptsyna and Le Breton \cite{KL01} and Hu and Nualart \cite{HN10}, continuous-time estimators of the drift parameter in Langevin's equation are studied. Kleptsyna and Le Breton \cite{KL01} provide a maximum likelihood estimator, where the stochastic integral with respect to the solution to Equation (\ref{main_equation}) returns to an It\^o integral. In \cite{TV07}, Tudor and Viens extend this estimator to equations with a drift function depending linearly on the unknown parameter. On the maximum likelihood estimator in fDEs with multiplicative noise, see Mishura and Ralchenko \cite{MR14}. Hu and Nualart \cite{HN10} provide a least squares estimator, where the stochastic integral with respect to the solution of Equation (\ref{main_equation}) is taken in the sense of Skorokhod. In \cite{HNZ19}, Hu, Nualart and Zhou extend this estimator to equations with a drift function depending linearly on the unknown parameter. Tindel and Neuenkirch \cite{NT14} provide a discrete-time least squares type estimator defined by an objective function allowing to make use of the main result of Tudor and Viens \cite{TV09} on the rate of convergence of the quadratic variation of the fractional Brownian motion. In \cite{PTV20}, Panloup, Tindel and Varvenne extend the results of \cite{NT14} under much more flexible conditions.\\
More recently, nonparametric methods were investigated to estimate the drift function $b$ in Equation (\ref{main_equation}). For instance, Saussereau \cite{SAUSSEREAU14} and Comte and Marie \cite{CM19} study the consistency of continuous-time Nadaraya-Watson type estimators of $b$.
\\
\\
The common point of all the references mentioned above is that the existence and uniqueness of the stationary solution to Equation (\ref{main_equation}) is required. Even if to estimate the distribution of the stationary solution is not necessary to study estimators of $b$, this is a very important question already investigated via kernel based methods in Bertin et al. \cite{BKPV20}. Precisely, when the stationary solution has a density $f$ with respect to Lebesgue's measure, the authors establish a risk bound on Parzen's estimator of $f$ and provide an oracle inequality for an adaptive estimator obtained via a Goldenshluger-Lepski type method. This is a nice application of a powerful concentration inequality for Lipschitz functionals of discrete samples of $X$ also established by Bertin et al. in \cite{BKPV20}.
\\
\\
Let $\mathcal S_m$ be the vector space generated by an orthonormal family $\mathcal B_m =\{\varphi_1,\dots,\varphi_m\}$ of $\mathbb L^2(I,dr)$, where $I\subset\mathbb R$ is an interval, and consider $t_0,\dots,t_n > 0$ such that $t_i := i\Delta_n$ for every $i\in\{1,\dots,n\}$, $\Delta_n > 0$,
\begin{displaymath}
\lim_{n\rightarrow\infty}\Delta_n = 0
\quad {\rm and}\quad
\lim_{n\rightarrow\infty}n\Delta_n =\infty.
\end{displaymath}
Our paper deals with the following projection estimator of $f$:
\begin{displaymath}
\widehat f_m =\widehat f_{m,n} :=
\sum_{j = 1}^{m}[\widehat\theta_{m,n}]_j\varphi_j
\quad {\rm with}\quad
\widehat\theta_{m,n} :=
\left(\frac{1}{n}\sum_{i = 1}^{n}\varphi_j(X_{t_i})\right)_{j\in\{1,\dots,m\}}.
\end{displaymath}
Precisely, a concentration inequality on the supremum of the empirical process is derived from Bertin et al. \cite{BKPV20}, Theorem 1 in Section \ref{section_preliminaries}. Then, a risk bound on $\widehat f_m$ is established in Section \ref{section_risk_bound}. Section \ref{section_model_selection} deals with a model selection method and an oracle inequality for the adaptive estimator. Finally, some basic numerical experiments are provided at Section \ref{numerical_section}.
\\
As in the {\it i.i.d. context}, the main advantage of the projection based approach for one-dimensional fDEs is that the optimization problem defining the model selection method is numerically easier to solve than the one defining the Goldenshluger-Lepski method because it involves only one variable. Note also that in the kernel based approach, even in dimension $1$ in the {\it i.i.d. context}, there is no simple model selection method as (\ref{model_selection_problem}). This is the main advantage of the adaptive estimator studied in this paper with respect to the kernel-based adaptive estimator of Bertin et al. \cite{BKPV20}. In fact, the same advantage than in the {\it i.i.d. context} (see the remarks at the end of Section \ref{section_model_selection} for details).
\\
\\
\textbf{Notations:}
\begin{enumerate}
 \item Throughout the paper, $\mathbb R^n$ is equipped with the distance $d_1$ defined by
 \begin{displaymath}
 d_1(\mathbf x,\mathbf y) :=
 \sum_{i = 1}^{n}|\mathbf y_i -\mathbf x_i|
 \textrm{ $;$ }
 \forall\mathbf x,\mathbf y\in\mathbb R^n.
 \end{displaymath}
 \item For any metric space $E$, ${\rm Lip}(E;\mathbb R)$ is the space of Lipschitz continuous maps from $E$ into $\mathbb R$, equipped with its usual semi-norm, always denoted by $\|.\|_{\rm Lip}$ for the sake of simplicity.
 \item The space $C^0(E;\mathbb R)$ is equipped with the uniform norm, always denoted by $\|.\|_{\infty}$ for the sake of simplicity.
 \item The space $\mathbb L^2(I,dr)$ is equipped with its usual scalar product $\langle .,.\rangle$. The associated norm is denoted by $\|.\|$.
 \item Throughout the paper, $\mathbb R^I$ is identified to $\{\varphi :\mathbb R\rightarrow\mathbb R :{\rm supp}(\varphi) = I\}$.
\end{enumerate}
%

% Section : Preliminaries: stationary solutions of the fDE.

%
\section{Preliminaries: stationary solutions of the fDE}\label{section_preliminaries}
This section deals with existing results on the existence and uniqueness of the stationary solution to Equation (\ref{main_equation}) under a dissipativity condition on the drift function $b$, and then with some consequences of a concentration inequality for Lipschitz functionals of $(X_{t_1},\dots,X_{t_n})$ due to Bertin et al. \cite{BKPV20}.
%

% Assumption : Assumption on b.

%
\begin{assumption}\label{assumption_b}
The function $b$ belongs to $C^1(\mathbb R)$, $b'$ is bounded and there exists $\mathfrak m_b > 0$ such that
\begin{displaymath}
b'(x)\leqslant -\mathfrak m_b
\textrm{ $;$ }
\forall x\in\mathbb R.
\end{displaymath}
\end{assumption}
\noindent
Under Assumption \ref{assumption_b}, Equation (\ref{main_equation}) has a unique stationary solution $X = (X_t)_{t\in\mathbb R_+}$ (see Hairer \cite{HAIRER05}) and $X_0$ has a density $f$ with respect to Lebesgue's measure (see Bertin et al. \cite{BKPV20}, Proposition 1).
%

% Example : Fractional Ornstein-Uhlenbeck process.

%
\begin{example}\label{fOU}
A famous example of fDE satisfying Assumption \ref{assumption_b} is the fractional Langevin equation ($b = -\theta{\rm Id}_{\mathbb R}$ with $\theta > 0$). Its solution is the fractional Ornstein-Uhlenbeck process and, in this case, the stationary density $f$ is Gaussian (see Cheridito et al. \cite{CKM03}).
\end{example}
\noindent
The following theorem provides a concentration inequality for Lipschitz functionals of $(X_{t_1},\dots,X_{t_n})$ (see Bertin et al. \cite{BKPV20}, Theorem 1).
%

% Theorem : Bertin et al. inequality for fDE.

%
\begin{theorem}\label{BKPV_inequality_fDE}
Under Assumption \ref{assumption_b}, there exists a constant $\mathfrak c_{\ref{BKPV_inequality_fDE}} > 0$, not depending on $n$, such that for every $F\in {\rm Lip}(\mathbb R^n;\mathbb R)$ and $r > 0$,
\begin{displaymath}
\mathbb P(F(X_{t_1},\dots,X_{t_n}) -\mathbb E(F(X_{t_1},\dots,X_{t_n})) > r)
\leqslant
\exp\left(
-\frac{r^2}{\mathfrak c_{\ref{BKPV_inequality_fDE}}\|F\|_{\rm Lip}^{2}n^{\mathfrak a_H}\Delta_{n}^{-\mathfrak b_H}}
\right)
\end{displaymath}
with $\mathfrak a_H = (2H)\vee 1$ and $\mathfrak b_H = 1\wedge (2 - 2H)$.
\end{theorem}
\noindent
Now, let us state a consequence of Theorem \ref{BKPV_inequality_fDE} on the (centered) empirical process, already established in Bertin et al. \cite{BKPV20} (see Corollary 1).
%

% Corollary : Corollary of Bertin et al. inequality for fDE.

%
\begin{corollary}\label{BKPV_inequality_fDE_corollary}
Under Assumption \ref{assumption_b}, for every $\varphi\in {\rm Lip}(\mathbb R)$ and $r > 0$,
\begin{displaymath}
\mathbb P\left(\frac{1}{n}\sum_{i = 1}^{n}[\varphi(X_{t_i}) -\mathbb E(\varphi(X_{t_i}))] > r\right)
\leqslant
\exp\left(-\frac{r^2}{\mathfrak c_{\ref{BKPV_inequality_fDE}}\|\varphi\|_{\rm Lip}^{2}}
(n\Delta_n)^{\mathfrak b_H}
\right).
\end{displaymath}
\end{corollary}
\noindent
Finally, because Corollary \ref{BKPV_inequality_fDE_corollary} is not sufficient to establish a risk bound on the adaptive estimator at Section \ref{section_model_selection}, let us derive a concentration inequality on the supremum of the empirical process from Theorem \ref{BKPV_inequality_fDE} under the following assumption of $f$.
%

% Assumption : Assumption on f.

%
\begin{assumption}\label{assumption_f}
The stationary density $f$ belongs to $\mathbb L^2(\mathbb R,dr)$.
\end{assumption}
\noindent
{\bf Remark.} It is plausible that $f$ fulfills Assumption \ref{assumption_f} for a wide class of drift functions fulfilling Assumption \ref{assumption_b}, but this problem is out of the scope of the present paper. However, let us provide some examples. On the one hand, since it is Gaussian, $f$ fulfills Assumption \ref{assumption_f} when (\ref{main_equation}) is the fractional Langevin equation ($b = -\theta{\rm Id}_{\mathbb R}$ with $\theta > 0$). On the other hand, when $H = 1/2$ and $\sigma > 0$, as mentioned in Bertin et al. \cite{BKPV20} (see Remark 1),
\begin{displaymath}
f(x) =\mathfrak c_{U,\sigma}\exp\left(\frac{U(x)}{2\sigma^2}\right)
\textrm{ $;$ }
\forall x\in\mathbb R
\end{displaymath}
where $U$ is a primitive function of $b$ and $\mathfrak c_{U,\sigma}$ is a positive constant. By Assumption \ref{assumption_b}, there exist $\mathfrak c_1,\mathfrak c_2 > 0$ such that
\begin{displaymath}
U(x)\leqslant -\mathfrak m_bx^2 +\mathfrak c_1x +\mathfrak c_2
\textrm{ $;$ }
\forall x\in\mathbb R.
\end{displaymath}
Then, $f$ fulfills Assumption \ref{assumption_f}.
%

% Corollary : Talagrand's inequality for fDE.

%
\begin{corollary}\label{Talagrand_inequality_fDE}
Consider $\mathcal F\subset {\rm Lip}(\mathbb R)$ such that
\begin{displaymath}
\mathfrak c_{\mathcal F} :=
\sup_{\varphi\in\mathcal F}\|\varphi\|_{\rm Lip} <\infty.
\end{displaymath}
Under Assumptions \ref{assumption_b} and \ref{assumption_f}, for every $\varphi\in\mathcal F$, consider the empirical process
\begin{displaymath}
\nu_n(\varphi) :=
\frac{1}{n}\sum_{i = 1}^{n}[\varphi(X_{t_i}) -\mathfrak m(\varphi)]
\quad\textrm{with}\quad
\mathfrak m(\varphi) :=\langle\varphi,f\rangle.
\end{displaymath}
Then, for every
\begin{displaymath}
\mathfrak h\geqslant
\mathbb E\left(\sup_{\varphi\in\mathcal F}|\nu_n(\varphi)|\right)
\end{displaymath}
and every $r > 0$,
\begin{displaymath}
\mathbb P\left(\sup_{\varphi\in\mathcal F}|\nu_n(\varphi)| -\mathfrak h > r\right)
\leqslant
\exp\left(
-\frac{r^2}{\mathfrak c_{\ref{BKPV_inequality_fDE}}\mathfrak c_{\mathcal F}^{2}}
(n\Delta_n)^{\mathfrak b_H}
\right).
\end{displaymath}
\end{corollary}
%

% Proof.

%
\begin{proof}
Let $F :\mathbb R^n\rightarrow\mathbb R$ be the map defined by
\begin{displaymath}
F(\mathbf x) :=
\sup_{\varphi\in\mathcal F}
\left|\frac{1}{n}\sum_{i = 1}^{n}[\varphi(\mathbf x_i) -\mathfrak m(\varphi)]\right|
\textrm{ $;$ }
\forall\mathbf x\in\mathbb R^n.
\end{displaymath}
For any $\mathbf x,\mathbf y\in\mathbb R^n$, by the triangle inequality for the uniform norm,
\begin{eqnarray*}
 |F(\mathbf y) - F(\mathbf x)|
 & \leqslant &
 \frac{1}{n}
 \sup_{\varphi\in\mathcal F}\left|
 \sum_{i = 1}^{n}[\varphi(\mathbf y_i) -\mathfrak m(\varphi)] -
 \sum_{i = 1}^{n}[\varphi(\mathbf x_i) -\mathfrak m(\varphi)]
 \right|\\
 & \leqslant &
 \frac{1}{n}
 \sup_{\varphi\in\mathcal F}\left\{
 \sum_{i = 1}^{n}|\varphi(\mathbf y_i) -\varphi(\mathbf x_i)|\right\}
 \leqslant
 \frac{1}{n}d_1(\mathbf x,\mathbf y)
 \sup_{\varphi\in\mathcal F}\|\varphi\|_{\rm Lip}.
\end{eqnarray*}
Then, $F\in {\rm Lip}(\mathbb R^n;\mathbb R)$ and
\begin{displaymath}
\|F\|_{\rm Lip}
\leqslant
\frac{\mathfrak c_{\mathcal F}}{n}.
\end{displaymath}
Therefore, by Theorem \ref{BKPV_inequality_fDE} and since $2 -\mathfrak a_H =\mathfrak b_H$, for every $r > 0$,
\begin{eqnarray*}
 \mathbb P\left(\sup_{\varphi\in\mathcal F}|\nu_n(\varphi)| -\mathfrak h > r\right)
 & \leqslant &
 \mathbb P(F(X_{t_1},\dots,X_{t_n}) -\mathbb E(F(X_{t_1},\dots,X_{t_n})) > r)\\
 & \leqslant &
 \exp\left(
 -\frac{r^2}{\mathfrak c_{\ref{BKPV_inequality_fDE}}\mathfrak c_{\mathcal F}^{2}}
 (n\Delta_n)^{\mathfrak b_H}
 \right).
\end{eqnarray*}
\end{proof}
%

% Section : Risk bound on the projection estimators.

%
\section{Risk bound on the projection estimators}\label{section_risk_bound}
This section deals with a risk bound on $\widehat f_m$ obtained via the concentration inequality on the empirical process stated in Corollary \ref{BKPV_inequality_fDE_corollary}.
\\
\\
In the sequel, $f_m$ is the orthogonal projection of $f$ on $\mathcal S_m$ (in $\mathbb L^2(I,dr)$), and $\mathcal B_m$ fulfills the following assumption.
%

% Assumption : First assumption on the basis.

%
\begin{assumption}\label{assumption_basis_1}
The $\varphi_j$'s are bounded and Lipschitz continuous functions.
\end{assumption}
\noindent
Now, let us establish the main result of this section: a risk bound on $\widehat f_m$.
%

% Proposition : Risk bound.

%
\begin{proposition}\label{risk_bound}
Under Assumptions \ref{assumption_b}, \ref{assumption_f} and \ref{assumption_basis_1},
\begin{displaymath}
\mathbb E(\|\widehat f_m - f\|^2)
\leqslant
\min_{t\in\mathcal S_m}\|t - f\|^2 +
2\mathfrak c_{\ref{BKPV_inequality_fDE}}\frac{mL(m)}{(n\Delta_n)^{\mathfrak b_H}}
\quad
\textrm{with}
\quad
L(m) :=\sum_{j = 1}^{m}\|\varphi_j\|_{\rm Lip}^{2}.
\end{displaymath}
\end{proposition}
%

% Proof.

%
\begin{proof}
First of all, for any $x\in I$,
\begin{displaymath}
\mathbb E(|\widehat f_m(x) - f(x)|^2) =
(\mathbb E(\widehat f_m(x)) - f(x))^2 +
{\rm var}(\widehat f_m(x)).
\end{displaymath}
Let us express well the bias term, and then give a suitable bound for the variance term. On the one hand,
\begin{displaymath}
\mathbb E(\widehat f_m(x)) =
\frac{1}{n}\sum_{j = 1}^{m}\sum_{i = 1}^{n}\mathbb E(\varphi_j(X_{t_i}))\varphi_j(x) =
\sum_{j = 1}^{m}\langle f,\varphi_j\rangle\varphi_j(x) = f_m(x).
\end{displaymath}
So,
\begin{displaymath}
\int_I
(\mathbb E(\widehat f_m(x)) - f(x))^2dx =
\|f_m - f\|^2 =
\min_{t\in\mathcal S_m}\|t - f\|^2.
\end{displaymath}
On the other hand, consider
\begin{displaymath}
\varphi(x,.) :=\sum_{j = 1}^{m}\varphi_j(.)\varphi_j(x).
\end{displaymath}
By Corollary \ref{BKPV_inequality_fDE_corollary} applied to $\varphi(x,.)$,
\begin{eqnarray*}
 {\rm var}(\widehat f_m(x))
 & = &
 2\int_{0}^{\infty}r\mathbb P\left(
 \frac{1}{n}\left|\sum_{i = 1}^{n}[\varphi(x,X_{t_i}) -\mathbb E(\varphi(x,X_{t_i}))]\right| > r
 \right)dr\\
 & \leqslant &
 4\int_{0}^{\infty}r
 \exp\left(-\frac{r^2}{\mathfrak c_{\ref{BKPV_inequality_fDE}}\|\varphi(x,.)\|_{\rm Lip}^{2}}
 (n\Delta_n)^{\mathfrak b_H}
 \right)dr =
 2\mathfrak c_{\ref{BKPV_inequality_fDE}}\Gamma(1)\frac{\|\varphi(x,.)\|_{\rm Lip}^{2}}{(n\Delta_n)^{\mathfrak b_H}}.
\end{eqnarray*}
Moreover, by Jensen's inequality and since $\mathcal B_m$ is an orthonormal family of $\mathbb L^2(I,dx)$,
\begin{displaymath}
\int_I
\|\varphi(x,.)\|_{\rm Lip}^{2}dx
\leqslant
\int_I\left[\sum_{j = 1}^{m}\|\varphi_j\|_{\rm Lip}|\varphi_j(x)|\right]^2dx
\leqslant m\sum_{j = 1}^{m}\|\varphi_j\|_{\rm Lip}^{2} =
mL(m).
\end{displaymath}
Therefore,
\begin{displaymath}
\mathbb E(\|\widehat f_m - f\|^2)
\leqslant
\min_{t\in\mathcal S_m}\|t - f\|^2 +
2\mathfrak c_{\ref{BKPV_inequality_fDE}}
\frac{mL(m)}{(n\Delta_n)^{\mathfrak b_H}}.
\end{displaymath}
\end{proof}
%

% Example : Example basis 1.

%
\begin{example}\label{example_basis_1}
Assume that $I = [0,1]$ and that $\mathcal B_m$ is the trigonometric basis. Precisely, for every $x\in I$,  $\varphi_1(x) = 1$, and for every $j\in\mathbb N$ such that $2j + 1\leqslant m$, $\varphi_{2j}(x) =\sqrt 2\cos(2\pi jx)$ and $\varphi_{2j + 1}(x) =\sqrt 2\sin(2\pi jx)$. Then, there exists a constant $\mathfrak c_{\ref{example_basis_1}} > 0$, not depending on $m$ and $n$, such that
\begin{displaymath}
L(m) =\sum_{j = 1}^{m}\|\varphi_j'\|_{\infty}^{2}\leqslant\mathfrak c_{\ref{example_basis_1}}m^3.
\end{displaymath}
So, the variance term in the risk bound on $\widehat f_m$ stated in Proposition \ref{risk_bound} is of order $m^4(n\Delta_n)^{-\mathfrak b_H}$. Note that the variance term in the risk bound on Parzen's estimator of bandwidth $h > 0$ obtained in Bertin et al. \cite{BKPV20}, Proposition 3, is of same order $h^{-4}(n\Delta_n)^{-\mathfrak b_H}$.
\end{example}
%

% Section : Model selection.

%
\section{Model selection}\label{section_model_selection}
As in the {\it classic} projection density estimation framework, note that
\begin{displaymath}
\widehat f_m =
\arg\min_{\varphi\in\mathcal S_m}\widehat\gamma_n(\varphi)
\quad {\rm with}\quad
\widehat\gamma_n(\varphi) :=
\|\varphi\|^2 -\frac{2}{n}\sum_{i = 1}^{n}\varphi(X_{t_i})
\textrm{ $;$ }
\forall\varphi\in\mathcal S_m.
\end{displaymath}
So, for the proposal set
\begin{displaymath}
\mathcal M :=
\{m\in\mathbb N^* :
mL(m)
\leqslant (n\Delta_n)^{\mathfrak b_H}\}
\quad {\rm with}\quad
\mathbb N^* =\mathbb N\backslash\{0\},
\end{displaymath}
\noindent
and a hyper-parameter $\mathfrak K > 0$, it is natural to consider the adaptive estimator $\widehat f_{\widehat m}$ of $f$, where
\begin{equation}\label{model_selection_problem}
\widehat m =
\arg\min_{m\in\mathcal M}\{\widehat\gamma_n(\widehat f_m) + {\rm pen}(m)\}
\end{equation}
with
\begin{displaymath}
{\rm pen}(m) :=
\mathfrak K
\frac{(m + 1)L(m)}{(n\Delta_n)^{\mathfrak b_H}}
\textrm{ $;$ }
\forall m\in\mathcal M.
\end{displaymath}
In order to provide an oracle inequality for $\widehat f_{\widehat m}$ at Theorem \ref{oracle_inequality}, let us first establish the following technical lemma.
%

% Lemma : Lemma for the oracle inequality.

%
\begin{lemma}\label{lemma_oracle_inequality}
Under Assumption \ref{assumption_basis_1}, with the notations of Corollary \ref{Talagrand_inequality_fDE}, if
\begin{displaymath}
\mathcal F =\mathcal F_m :=
\{\varphi\in\mathcal S_m :\|\varphi\| = 1\},
\end{displaymath}
then
\begin{displaymath}
\mathfrak c_{\mathcal F_m}\leqslant L(m)^{1/2}.
\end{displaymath}
\end{lemma}
%

% Proof.

%
\begin{proof}
Consider $\varphi\in\mathcal S_m$ such that $\|\varphi\| = 1$. Then, there exist $a_1,\dots,a_m\in\mathbb R$ such that
\begin{displaymath}
\varphi =\sum_{j = 1}^{m}a_j\varphi_j
\quad {\rm and}\quad
\sum_{j = 1}^{m}a_{j}^{2} = 1.
\end{displaymath}
So,
\begin{displaymath}
\|\varphi\|_{\rm Lip}
\leqslant\sum_{j = 1}^{m}|a_j|\cdot\|\varphi_j\|_{\rm Lip}
\leqslant\left(\sum_{j = 1}^{m}a_{j}^{2}\right)^{1/2}
\left(\sum_{j = 1}^{m}\|\varphi_j\|_{\rm Lip}^{2}\right)^{1/2} =
L(m)^{1/2}.
\end{displaymath}
Therefore,
\begin{displaymath}
\mathfrak c_{\mathcal F_m} =
\sup_{\varphi\in\mathcal F_m}\|\varphi\|_{\rm Lip}
\leqslant
L(m)^{1/2}.
\end{displaymath}
\end{proof}
\noindent
In the sequel, the $\mathcal S_m$'s are nested:
%

% Assumption : Assumption on the basis.

%
\begin{assumption}\label{assumption_basis_2}
For every $m,m'\in\mathcal M$, if $m\geqslant m'$, then $\mathcal S_{m'}\subset\mathcal S_m$.
\end{assumption}
\noindent
Moreover, $L(.)$ fulfills the following assumption.
%

% Assumption : Assumption on L.

%
\begin{assumption}\label{assumption_L}
The map $m\mapsto L(m)$ has polynomial growth.
\end{assumption}
%

% Example : Example basis 2.

%
\begin{example}\label{example_basis_2}
On the one hand, note that Assumption \ref{assumption_basis_2} is fulfilled by several usual bases: the trigonometric basis, Hermite's basis, Laguerre's basis, etc. On the other hand, for instance, the trigonometric basis fulfills Assumption \ref{assumption_L} (see Example \ref{example_basis_1}).
\end{example}
\noindent
Now, let us establish the main result of this section: an oracle inequality for $\widehat f_{\widehat m}$.
%

% Theorem : Oracle inequality.

%
\begin{theorem}\label{oracle_inequality}
Under Assumptions \ref{assumption_b}, \ref{assumption_f}, \ref{assumption_basis_1}, \ref{assumption_basis_2} and \ref{assumption_L}, if $\mathfrak K\geqslant 16\mathfrak c_{\ref{BKPV_inequality_fDE}}$, then there exist two positive constants $\mathfrak c_{\ref{oracle_inequality},1}$ and $\mathfrak c_{\ref{oracle_inequality},2}$, not depending on $n$, such that
\begin{displaymath}
\mathbb E(\|\widehat f_{\widehat m} - f\|^2)
\leqslant
\mathfrak c_{\ref{oracle_inequality},1}\min_{m\in\mathcal M}
\{\|f_m - f\|^2 + {\rm pen}(m)\}
+\frac{\mathfrak c_{\ref{oracle_inequality},2}}{(n\Delta_n)^{\mathfrak b_H}}.
\end{displaymath}
\end{theorem}
%

% Proof.

%
\begin{proof}
The proof is dissected in two steps. In the first one, with the same arguments than in the {\it classic} projection density estimation framework (see Comte \cite{COMTE19}, Theorem 5.2 or Massart \cite{MASSART03}, Chapter 7), it is established that for any $m\in\mathcal M$,
\begin{displaymath}
\|\widehat f_{\widehat m} - f\|^2
\leqslant
3\|f_m - f\|^2 + 4{\rm pen}(m) + R_{m,n},
\end{displaymath}
where $R_{m,n}$ is a remainder term. In the second step, it is established that $\mathbb E(R_{m,n})$ is of order $(n\Delta_n)^{-\mathfrak b_H}$ thanks to Corollary \ref{Talagrand_inequality_fDE} and Lemma \ref{lemma_oracle_inequality}.
\\
\\
{\bf Step 1.} Note that
\begin{displaymath}
\widehat\gamma_n(\widehat f_{\widehat m}) + {\rm pen}(\widehat m)
\leqslant
\widehat\gamma_n(f_m) + {\rm pen}(m)
\textrm{ $;$ }
\forall m\in\mathcal M
\end{displaymath}
and
\begin{displaymath}
\widehat\gamma_n(\varphi) -\widehat\gamma_n(\psi) =
\|\varphi - f\|^2 -\|\psi - f\|^2 - 2\nu_n(\varphi -\psi)
\textrm{ $;$ }
\forall\varphi,\psi\in {\rm Lip}(I;\mathbb R).
\end{displaymath}
Then, for any $m\in\mathcal M$, since $2uv\leqslant u^2 + v^2$ for every $u,v\in\mathbb R_+$, and since
\begin{displaymath}
\mathcal S_m +\mathcal S_{\overline m} =
\left\{\sum_{j = 1}^{m}\theta_j\varphi_j +
\sum_{j = 1}^{\overline m}\overline\theta_j\varphi_j
\textrm{ $;$ }
\theta_1,\dots,\theta_m,\overline\theta_1,\dots,\overline\theta_{\overline m}\in\mathbb R\right\}
\subset\mathcal S_{m\vee\overline m}
\textrm{ $;$ }
\forall\overline m\in\mathcal M
\end{displaymath}
by Assumption \ref{assumption_basis_2},
\begin{eqnarray*}
 \|\widehat f_{\widehat m} - f\|^2
 & \leqslant &
 \|f_m - f\|^2 + {\rm pen}(m) +
 2\cdot\frac{1}{2}\|\widehat f_{\widehat m} - f_m\|\cdot
 2\nu_n\left(\frac{\widehat f_{\widehat m} - f_m}{\|\widehat f_{\widehat m} - f_m\|}\right) -
 {\rm pen}(\widehat m)\\
 & \leqslant &
 \|f_m - f\|^2 + {\rm pen}(m) +
 \frac{1}{4}\|\widehat f_{\widehat m} - f_m\|^2 +
 4\left[\sup_{\varphi\in\mathcal F_{m\vee\widehat m}}|\nu_n(\varphi)|\right]^2 -
 {\rm pen}(\widehat m).
\end{eqnarray*}
For every $\overline m\in\mathcal M$, consider
\begin{eqnarray*}
 p(m,\overline m) & := &
 \frac{\mathfrak K}{4}\cdot
 \frac{((m\vee\overline m) + 1)L(m\vee\overline m)}{(n\Delta_n)^{\mathfrak b_H}}\\
 & \leqslant &
 \frac{1}{4}({\rm pen}(m) + {\rm pen}(\overline m)).
\end{eqnarray*}
So,
\begin{eqnarray*}
 \|\widehat f_{\widehat m} - f\|^2
 & \leqslant &
 \|f_m - f\|^2 + {\rm pen}(m) +\frac{1}{4}\|\widehat f_{\widehat m} - f_m\|^2\\
 & &
 \hspace{3cm} +
 4\left(\left[\sup_{\varphi\in\mathcal F_{m\vee\widehat m}}|\nu_n(\varphi)|\right]^2
 - p(m,\widehat m)\right)_+ +
 4p(m,\widehat m) - {\rm pen}(\widehat m)\\
 & \leqslant &
 \|f_m - f\|^2 + 2{\rm pen}(m) +
 \frac{1}{2}(\|\widehat f_{\widehat m} - f\|^2 +\|f_m - f\|^2)\\
 & &
 \hspace{3cm} +
 4\sum_{\overline m\in\mathcal M}
 \left(\left[\sup_{\varphi\in\mathcal F_{m\vee\overline m}}|\nu_n(\varphi)|\right]^2
 - p(m,\overline m)\right)_+
\end{eqnarray*}
and then,
\begin{displaymath}
\|\widehat f_{\widehat m} - f\|^2
\leqslant
3\|f_m - f\|^2 + 4{\rm pen}(m) +
8\sum_{\overline m\in\mathcal M}
 \left(\left[\sup_{\varphi\in\mathcal F_{m\vee\overline m}}|\nu_n(\varphi)|\right]^2
 - p(m,\overline m)\right)_+.
\end{displaymath}
{\bf Step 2.} For any $\overline m\in\mathcal M$, by Corollary \ref{BKPV_inequality_fDE_corollary} (as in the proof of Proposition \ref{risk_bound}), and since $\mathfrak K\geqslant 16\mathfrak c_{\ref{BKPV_inequality_fDE}}$,
\begin{eqnarray*}
 \mathbb E\left(\sup_{\varphi\in\mathcal F_{m\vee\overline m}}|\nu_n(\varphi)|\right)^2
 & \leqslant &
 \sum_{j = 1}^{m\vee\overline m}\mathbb E(\nu_n(\varphi_j)^2) =
 \sum_{j = 1}^{m\vee\overline m}{\rm var}(\nu_n(\varphi_j))\\
 & \leqslant &
 2\mathfrak c_{\ref{BKPV_inequality_fDE}}
 \frac{1}{(n\Delta_n)^{\mathfrak b_H}}
 \sum_{j = 1}^{m\vee\overline m}\|\varphi_j\|_{\rm Lip}^{2}
 \leqslant
 \frac{\mathfrak K}{8}\cdot
 \frac{L(m\vee\overline m)}{(n\Delta_n)^{\mathfrak b_H}}
 =:\mathfrak h(m,\overline m).
\end{eqnarray*}
Now, consider
\begin{displaymath}
\mathfrak p(m,\overline m) :=
\frac{1}{2}p(m,\overline m) -\mathfrak h(m,\overline m) =
\frac{\mathfrak K}{8}\cdot
\frac{(m\vee\overline m)L(m\vee\overline m)}{(n\Delta_n)^{\mathfrak b_H}} > 0.
\end{displaymath}
So,
\begin{eqnarray*}
 \left(\left[\sup_{\varphi\in\mathcal F_{m\vee\overline m}}|\nu_n(\varphi)|\right]^2
 - p(m,\overline m)\right)_+
 & \leqslant &
 \left(2\left[\sup_{\varphi\in\mathcal F_{m\vee\overline m}}|\nu_n(\varphi)| -\mathfrak h(m,\overline m)^{1/2}\right]^2 +
 2\mathfrak h(m,\overline m) - p(m,\overline m)\right)_+\\
 & = &
 2\left(\left[\sup_{\varphi\in\mathcal F_{m\vee\overline m}}|\nu_n(\varphi)| -\mathfrak h(m,\overline m)^{1/2}\right]^2 -\mathfrak p(m,\overline m)\right)_+
\end{eqnarray*}
and then,
\begin{eqnarray*}
 & &
 \mathbb E\left[\left(\left[\sup_{\varphi\in\mathcal F_{m\vee\overline m}}|\nu_n(\varphi)|\right]^2
 - p(m,\overline m)\right)_+
 \right]\\
 & &
 \hspace{3cm}\leqslant
 2\int_{0}^{\infty}
 \mathbb P\left[\left(\left[\sup_{\varphi\in\mathcal F_{m\vee\overline m}}|\nu_n(\varphi)| -\mathfrak h(m,\overline m)^{1/2}\right]^2 -\mathfrak p(m,\overline m)\right)_+ > r\right]dr\\
 & &
 \hspace{3cm} =
 2\int_{0}^{\infty}
 \mathbb P\left(\sup_{\varphi\in\mathcal F_{m\vee\overline m}}|\nu_n(\varphi)| -
 \mathfrak h(m,\overline m)^{1/2} > (r +\mathfrak p(m,m'))^{1/2}\right)dr.
\end{eqnarray*}
Thus, by Corollary \ref{Talagrand_inequality_fDE},
\begin{eqnarray*}
 \mathbb E\left[\left(\left[\sup_{\varphi\in\mathcal F_{m\vee\overline m}}|\nu_n(\varphi)|\right]^2
 - p(m,\overline m)\right)_+
 \right]
 & \leqslant &
 2\exp\left(
 -\frac{\mathfrak p(m,\overline m)(n\Delta_n)^{\mathfrak b_H}}{\mathfrak c_{\ref{BKPV_inequality_fDE}}\mathfrak c_{\mathcal F_{m\vee\overline m}}^{2}}\right)
 \int_{0}^{\infty}\exp\left(
 -\frac{r(n\Delta_n)^{\mathfrak b_H}}{\mathfrak c_{\ref{BKPV_inequality_fDE}}\mathfrak c_{\mathcal F_{m\vee\overline m}}^{2}}\right)dr\\
 & = &
 2\Gamma(1)
 \mathfrak c_{\ref{BKPV_inequality_fDE}}
 \frac{1}{(n\Delta_n)^{\mathfrak b_H}}
 \mathfrak c_{\mathcal F_{m\vee\overline m}}^{2}
 \exp\left(
 -\frac{\mathfrak p(m,\overline m)(n\Delta_n)^{\mathfrak b_H}}{\mathfrak c_{\ref{BKPV_inequality_fDE}}\mathfrak c_{\mathcal F_{m\vee\overline m}}^{2}}\right).
\end{eqnarray*}
Moreover, by Assumption \ref{assumption_L}, there exist three constants $\mathfrak c_1,\mathfrak c_2,\mathfrak c_3 > 0$, not depending on $m$ and $n$, such that
\begin{eqnarray*}
 \sum_{\overline m\in\mathcal M}
 \mathfrak c_{\mathcal F_{m\vee\overline m}}^{2}\exp\left(
 -\frac{\mathfrak p(m,\overline m)(n\Delta_n)^{\mathfrak b_H}}{\mathfrak c_{\ref{BKPV_inequality_fDE}}\mathfrak c_{\mathcal F_{m\vee\overline m}}^{2}}\right)
 & \leqslant &
 \sum_{\overline m\in\mathcal M}
 L(m\vee\overline m)\exp\left(
 -\frac{\mathfrak K}{8\mathfrak c_{\ref{BKPV_inequality_fDE}}}(m\vee\overline m)\right)\\
 & = &
 mL(m)\exp\left(
 -\frac{\mathfrak K}{8\mathfrak c_{\ref{BKPV_inequality_fDE}}}m\right)\\
 & &
 \hspace{1cm}
 +\sum_{\overline m\in\mathcal M\textrm{ $:$ }\overline m > m}
 L(\overline m)\exp\left(
 -\frac{\mathfrak K}{16\mathfrak c_{\ref{BKPV_inequality_fDE}}}\overline m\right)
 \exp\left(
 -\frac{\mathfrak K}{16\mathfrak c_{\ref{BKPV_inequality_fDE}}}\overline m\right)\\
 & \leqslant &
 \mathfrak c_1 +
 \mathfrak c_2\sum_{\overline m\in\mathcal M}\exp\left(
 -\frac{\mathfrak K}{16\mathfrak c_{\ref{BKPV_inequality_fDE}}}\overline m\right)
 \leqslant\mathfrak c_3.
\end{eqnarray*}
Therefore,
\begin{displaymath}
\mathbb E(\|\widehat f_{\widehat m} - f\|^2)
\leqslant
\min_{m\in\mathcal M}
\{3\|f_m - f\|^2 + 4{\rm pen}(m)\}
+\frac{2\mathfrak c_{\ref{BKPV_inequality_fDE}}\mathfrak c_3}{(n\Delta_n)^{\mathfrak b_H}}.
\end{displaymath}
\end{proof}
\noindent
{\bf Remarks:}
\begin{enumerate}
 \item Note that the penalty and the remainder term in the risk bound in Theorem \ref{oracle_inequality} are of same order than in Bertin et al. \cite{BKPV20}, Theorem 3.
 \item As in the {\it i.i.d. context}, since it involves only one variable, the main advantage of the projection based approach for one-dimensional fDEs is that Problem (\ref{model_selection_problem}) is numerically easier to solve than the optimization problem defining the Goldenshluger-Lepski method. Note also that in the kernel based approach, even in dimension $1$ in the {\it i.i.d. context}, there is no simple model selection method as (\ref{model_selection_problem}). Recently, in \cite{LMR17}, Lacour, Massart and Rivoirard have provided a bandwidth selection method, called PCO method, bypassing the mentioned drawbacks of Goldenshluger-Lepski's method, but to extend the PCO method to the fDE framework requires more than Bertin et al. \cite{BKPV20}, Theorem 1. An extension of the concentration inequality for $U$-statistics of order $2$ of Houdr\'e and Reynaud-Bouret \cite{HRB03} from the {\it i.i.d. context} to $(X_{t_1},\dots,X_{t_n})$ is required.
 \item Note that without additional arguments, one can extend the result of our paper to multidimensional fDEs for isotropic projection estimators. However, to extend it to anisotropic projection estimators requires a Goldenshluger-Lepski type procedure again (see Chagny \cite{CHAGNY13} for this type of method).
\end{enumerate}
%

% Section : Basic numerical experiments.

%
\section{Basic numerical experiments}\label{numerical_section}
In this section, for $H\in [1/2,1)$, our projection estimator of the stationary density $f$ is evaluated numerically when (\ref{main_equation}) is the fractional Langevin equation ($b = -\theta{\rm Id}_{\mathbb R}$ with $\theta > 0$). In this case,
\begin{displaymath}
f(x) =
\frac{1}{\sqrt{2\pi\sigma^2\theta^{-2H}H\Gamma(2H)}}
\exp\left(-\frac{x^2}{2\sigma^2\theta^{-2H}H\Gamma(2H)}
\right)
\textrm{ $;$ }
\forall x\in\mathbb R.
\end{displaymath}
The fractional Brownian motion is simulated via the Decreusefond-Lavaud method (see Decreusefond and Lavaud \cite{DL96}) for $H = 0.5$ and $H = 0.7$ along the dissection $\{kT/n\textrm{ $;$ }k = 0,\dots,n\}$ of $[0,T]$ with $T = 100$ and $n = 10^3$. The fractional Langevin equation is simulated with the initial condition $x_0 = 5$, $\theta = 5$ and $\sigma = 1$ by using the step-$n$ Euler scheme $X^{(n)}$ defined by
\begin{displaymath}
\left\{
\begin{array}{rcl}
 X_{0}^{(n)} & = & x_0\\
 X_{k + 1}^{(n)} & = &
 X_{k}^{(n)} -\theta X_{k}^{(n)}T/n +
 \sigma (B_{t_{k + 1}} - B_{t_k})
 \textrm{ $;$ }
 k = 0,\dots,n - 1
\end{array}\right..
\end{displaymath}
Since under Assumption \ref{assumption_b} the solution to Equation (\ref{main_equation}) with initial condition $x_0$ converges pathwise and exponentially fast to its stationary solution when $t\rightarrow\infty$, for $T$ and $n$ large enough, the error induced by considering datasets generated by the Euler scheme $X^{(n)}$ is not that significant.\\
In each case ($H = 0.5$ or $H = 0.7$), for the trigonometric basis, our projection estimator is computed on 10 independent datasets along the dissection $\{j/N\textrm{ $;$ }j = -N,\dots,N\}$ of $[-1,1]$ with $N = 70$. The average MISE is provided.
\\
\\
{\bf Case $H = 0.5$.} On Figure \ref{projection_0.5}, the 5 estimations (dashed black curves) of $f$ (red curve) generated by $\widehat f_m$ are plotted for $m = 5$, leading to an average MISE of $0.013$.
\begin{figure}[!h]
\centering
\includegraphics[scale=0.5]{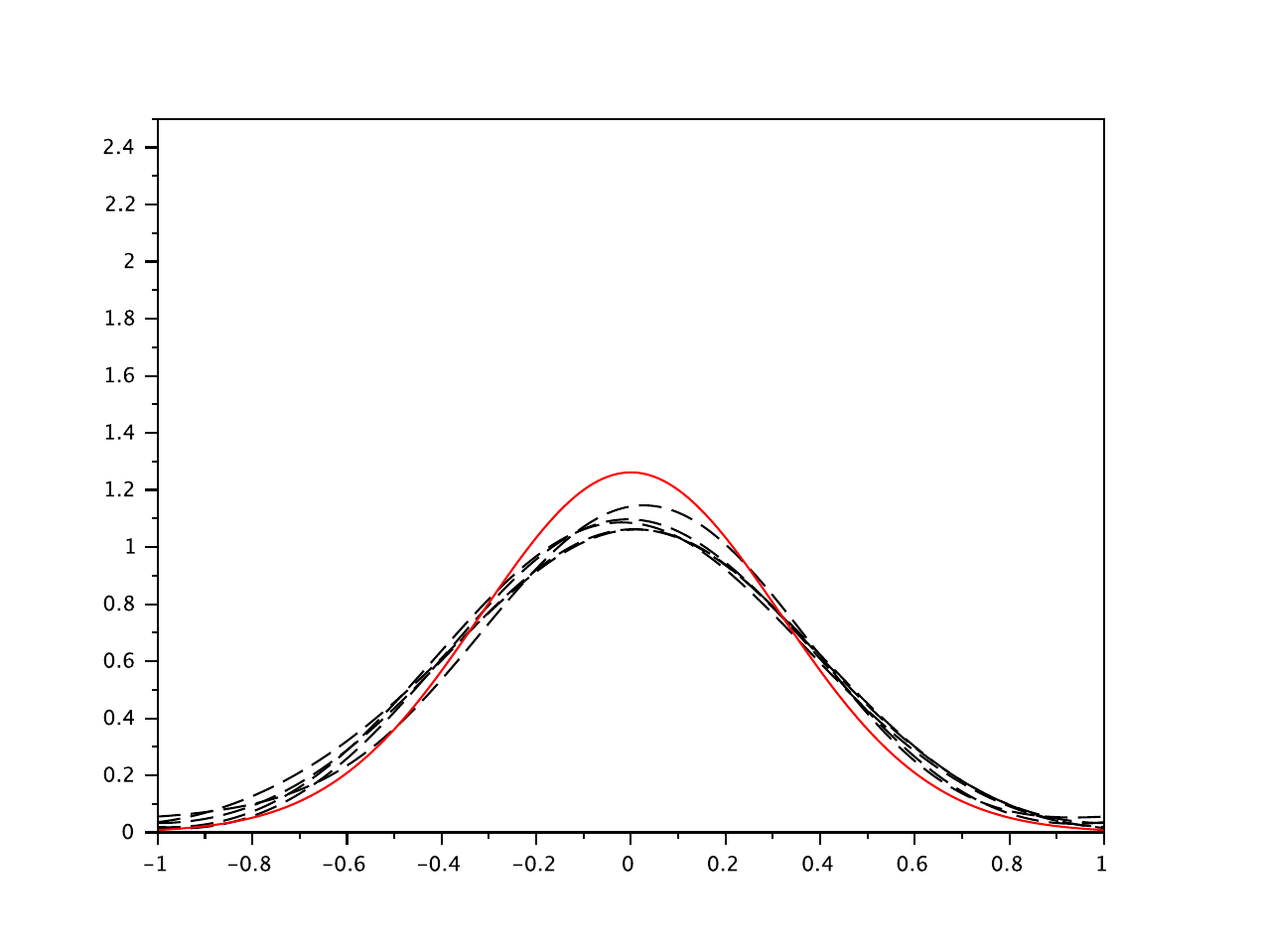} 
\caption{Projection estimations of $f$ computed for $m = 5$ on $5$ datasets with $H = 0.5$.}
\label{projection_0.5}
\end{figure}
\newline
{\bf Case $H = 0.7$.} On Figure \ref{projection_0.7}, the 5 estimations (dashed black curves) of $f$ (red curve) generated by $\widehat f_m$ are plotted for $m = 5$, leading to an average MISE of $0.044$.
\begin{figure}[!h]
\centering
\includegraphics[scale=0.5]{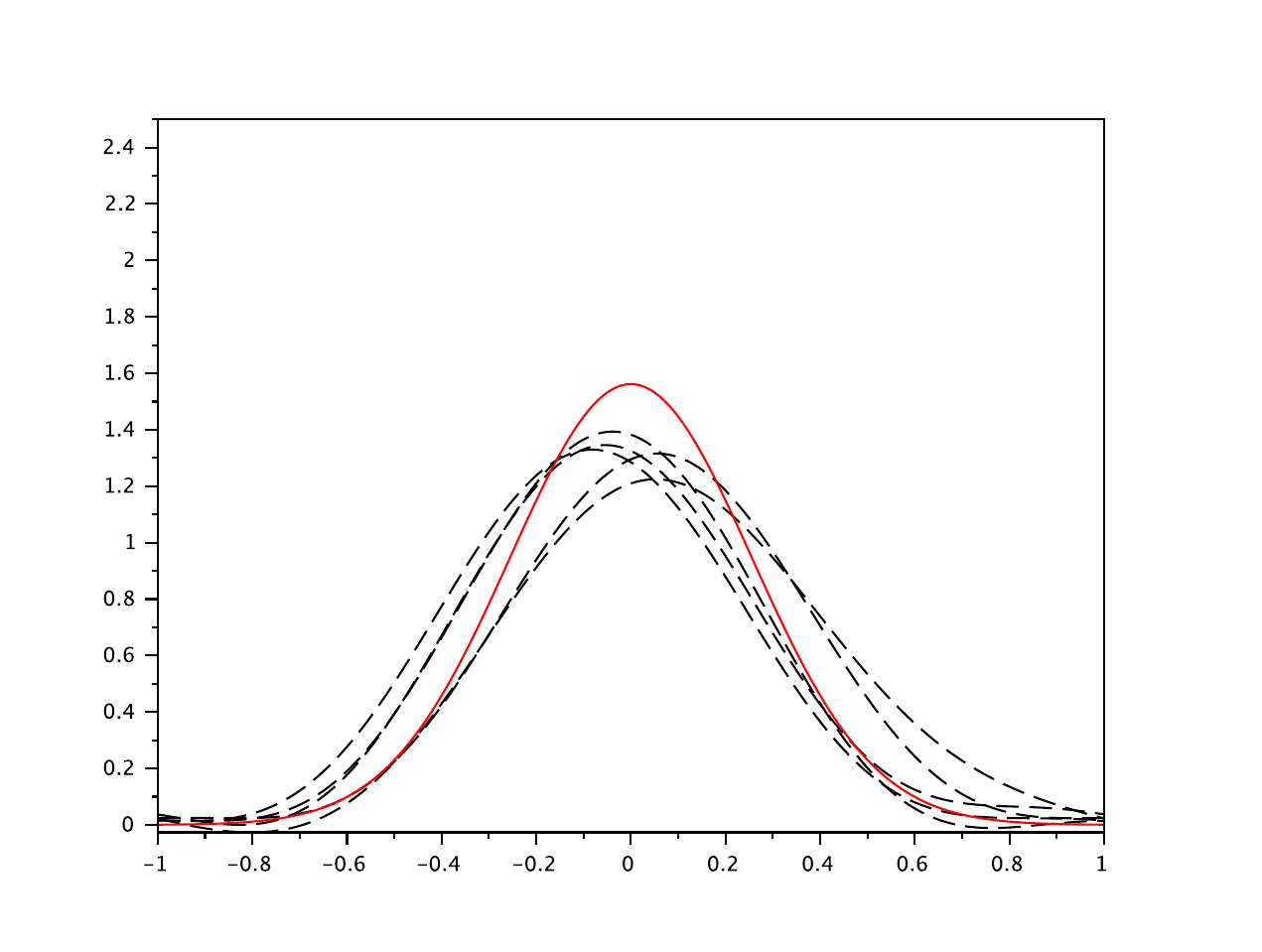} 
\caption{Projection estimations of $f$ computed for $m = 5$ on $5$ datasets with $H = 0.7$.}
\label{projection_0.7}
\end{figure}
\newline
\newline
{\bf Acknowledgments.} Thank you very much to H\'el\`ene Halconruy for her carful reading of this paper.
\newpage
%

% References.

%

%

\begin{thebibliography}{99}
 \bibitem{BKPV20} K. Bertin, N. Klutchnikoff, F. Panloup and M. Varvenne. {\it Adaptive Estimation of the Stationary Density of a Stochastic Differential Equation Driven by a Fractional Brownian Motion.} Statistical Inference for Stochastic Processes 23, 2, 271-300, 2020.
 \bibitem{CHAGNY13} G. Chagny. {\it Warped Bases for Conditional Density Estimation.} Mathematical Methods of Statistics 22, 253-282, 2013.
 \bibitem{CKM03} P. Cheridito, H. Kawaguchi and M. Maejima. {\it Fractional Ornstein-Uhlenbeck Process.} Electronic Journal of Probability 8, 3, 1-14, 2003.
 \bibitem{COMTE19} F. Comte. {\it Nonparametric Estimation.} Spartacus IDH, 2019.
 \bibitem{CM19} F. Comte and N. Marie {\it Nonparametric Estimation in Fractional SDE.} Statistical Inference for Stochastic Processes 22, 3, 359-382, 2019.
 \bibitem{DL96} L. Decreusefond and N. Lavaud. {\it Simulation of the Fractional Brownian Motion and Application to the Fluid Queue.} Proceedings of the ATNAC'96 conference, 1996.
 \bibitem{HAIRER05} M. Hairer. {\it Ergodicity of Stochastic Differential Equations Driven by Fractional Brownian Motion.} The Annals of Probability 33, 2, 703-758, 2005.
 \bibitem{HRB03} C. Houdr\'e and P. Reynaud-Bouret. {\it Exponential Inequalities, with Constants, for $U$-Statistics of Order Two.} Statistics Inequalities and Applications, Progr. Probab. 56, 55-69, 2003.
 \bibitem{HN10} Y. Hu and D. Nualart. {\it Parameter Estimation for Fractional Ornstein-Uhlenbeck Processes.} Statistics and Probability Letters 80, 1030-1038, 2010.
 \bibitem{HNZ19} Y. Hu, D. Nualart and H. Zhou. {\it Drift Parameter Estimation for Nonlinear Stochastic Differential Equations Driven by Fractional Brownian Motion.} Stochastic 91, 8, 1067-1091, 2019.
 \bibitem{KL01} M.L. Kleptsyna and A. Le Breton. {\it Some Explicit Statistical Results About Elementary Fractional Type Models.} Nonlinear Analysis 47, 4783-4794, 2001.
 \bibitem{LMR17} C. Lacour, P. Massart and V. Rivoirard. {\it Estimator Selection: a New Method with Applications to Kernel Density Estimation.} Sankhya A 79, 2, 298-335, 2017.
 \bibitem{MR14} J. Mishura and K. Ralchenko. {\it On Drift Parameter Estimation in Models with Fractional Brownian Motion by Discrete Observation.} Austrian Journal of Statistics 43, 3-4, 217-228, 2014.
 \bibitem{NT14} A. Neuenkirch and S. Tindel. {\it A Least Square-Type Procedure for Parameter Estimation in Stochastic Differential Equations with Additive Fractional Noise.} Statistical Inference for Stochastic Processes 17, 1, 99-120, 2014.
 \bibitem{MASSART03} P. Massart. {\it Concentration Inequalities and Model Selection.} Ecole d'Et\'e de Probabilit\'es de Staint-Flour XXXIII, 2003.
 \bibitem{PTV20} F. Panloup, S. Tindel and M. Varvenne. {\it A General Drift Estimation Procedure for Stochastic Differential Equations with Additive Fractional Noise.} Electronic Journal of Statistics 14, 1, 1075-1136, 2020.
 \bibitem{SAUSSEREAU14} B. Saussereau. {\it Nonparametric Inference for Fractional Diffusions.} Bernoulli 20, 2, 878-918, 2014.
 \bibitem{TV07} C.A. Tudor and F. Viens. {\it Statistical Aspects of the Fractional Stochastic Calculus.} The Annals of Statistics 35, 3, 1183-1212, 2007.
 \bibitem{TV09} C.A. Tudor and F. Viens. {\it Variations and Estimators of Self-Similarity Parameters via Malliavin Calculus.} The Annals of Probability 37, 6, 2093-2134, 2009.
\end{thebibliography}
\end{document}